\documentclass{amsart}
\usepackage{mathpazo}
\usepackage{amssymb}

\newtheorem{theorem}{Theorem}[section]

\newtheorem*{Acknowledgement}{\textnormal{\textbf{Acknowledgement}}}
\theoremstyle{definition}
\newtheorem{definition}[theorem]{Definition}
\newtheorem{example}[theorem]{Example}

\newtheorem{corollary}[theorem]{Corollary}
\newtheorem{proposition}[theorem]{Proposition}
\newtheorem{remark}[theorem]{Remark}

\numberwithin{equation}{section}
\newcommand{\beqa}{\begin{eqnarray*}}
	\newcommand{\eeqa}{\end{eqnarray*}}
\newcommand{\beqn}{\begin{eqnarray}}
	\newcommand{\eeqn}{\end{eqnarray}}

\newcommand{\ci}{\subseteq}

\renewcommand{\a}{\alpha}

\newcommand{\e}{\varepsilon}

\newcounter{cnt1}
\newcounter{cnt2}
\newcounter{cnt3}
\newcommand{\blr}{\begin{list}{$($\roman{cnt1}$)$}
		{\usecounter{cnt1} \setlength{\topsep}{0pt}
			\setlength{\itemsep}{0pt}}}
	\newcommand{\bla}{\begin{list}{$($\alph{cnt2}$)$}
			{\usecounter{cnt2} \setlength{\topsep}{0pt}
				\setlength{\itemsep}{0pt}}}
		\newcommand{\bln}{\begin{list}{$($\arabic{cnt3}$)$}
				{\usecounter{cnt3} \setlength{\topsep}{0pt}
					\setlength{\itemsep}{0pt}}}
			\newcommand{\el}{\end{list}}
		\newtheorem{thm}{Theorem}

		\newtheorem{Def}[thm]{Definition}
		
		\newtheorem{rem}[thm]{Remark}
		\newcommand{\Rem}{\begin{rem} \rm}
			\newcommand{\bdfn}{\begin{Def} \rm}
				\newcommand{\edfn}{\end{Def}}

			\title{Small Diameter  Properties In Ideals of Banach Spaces}
			\author[ S. Basu , S. Seal ]
			{Sudeshna Basu$^{1}$, Susmita Seal$ ^{2}$ }
			\address{{$^{1}$}   Sudeshna Basu,
				Department of Mathematics, 
				Ram Krishna Mission Vivekananda Education and Research Institute , 
				Belur Math,  Howrah 711202
				West Bengal, India and
				Department of Mathematics,
				George Washington University,
				Washington DC 20052 USA 
			}
			\email{sudeshnamelody@gmail.com}
			
			\address {{$^{2}$} Susmita Seal, 
				Department of Mathematics, ,
				Ram Krishna Mission Vivekananda Education and Research Institute , 
				Belur Math,  Howrah 711202,
				West Bengal, India}
			\email{susmitaseal1996@gmail.com}

			\date{}
\begin{document}
\maketitle
\begin{abstract}
	A Banach space has the ball huskable  property ($BHP$) if the closed unit ball  has weakly open sets  of arbitrarily small diameter. We can analogously define $w^*$-$BHP$  in the  dual space. In this short note, we study these properties in the context of  ideals in Banach spaces. The notion of an ideal,
	was introduced by Godefroy, Kalton and Saphar in \cite{GKS}. We show that if a Banach space $X$ has $BHP,$ then any 
	$M$-ideal of $X$ also has $BHP.$ We further show that if $Y$ is an $M$-ideal of $X,$ then $Y^*$ has $w^*$-$BHP$ implies $X^*$ has $w^{*}$-$BHP.$
	We use this result to prove that for a compact Hausdorff space $K$ which has an isolated point, $X$ has $BHP$ whenever $C(K,X)$ has $BHP$ and  $X^*$ has $w^{*}$-$BHP$  implies $C(K,X)^*$ has $w^{*}$-$BHP.$
	We also prove that $w^*$-$BHP$ can be lifted from $Y^*$ to $X^*$ provided $Y$ is a strict ideal of $X$.
	Lastly, we show that if $Y$ is an almost isometric ideal of $X,$ then $BHP$ can be lifted from $Y$ to $X.$
	We obtain similar results for ball dentable property ($BDP$) and ball small combination of slices Property ($BSCSP$) as well.
\end{abstract}

\section{Introduction}

Let $X$ be a {\it real} nontrivial Banach space and $X^*$ its dual. We will denote by $B_X$, $S_X$ and $B_X(x, r)$ the closed unit ball, the unit sphere and the closed ball of radius $r >0$ and center $x$. We refer to the monograph \cite{B1} for notions of convexity theory that we will be using here.

\bdfn
\blr
\item We say $A \ci B_{X^*}$ is a norming set for $X$ if $\|x\| =
\sup\{|x^*(x)| : x^* \in A\}$, for all $x \in X$. A closed subspace $F
\ci X^*$ is a norming subspace if $B_F$ is a norming set for $X$.
\item Let $f \in X^*$, $\a > 0$ and $ C \subseteq X$.
Then the set $S(C, f, \a) = \{x \in C : f(x) > \mbox{sup}~ f(C) - \a \}$ is called the open slice determined by $f$ and
$\a.$  We assume without loss of generality that $\|f\| = 1$. One can analogously define $w^*$ slices in $X^*$.
%\item
%A point $x \neq 0$ in a convex set $C \ci X$ is called a denting point of $K$, if for every $\e > 0$, there exist slices$S$ of $K$, such that $x \in S$ and $diam(S) <\e.$ One can analogously define $w^*$-denting point in $X^*$.
%\item A point $x \neq 0$ in a convex set $K \ci X$ is called a SCS small combination of slices (SCS)point of $K$, if for every $\e > 0$, there exist slice$S_i$ of $K$, and a {\em convex} combination $S =\sum_{i=1}^{n} \la_i S_i,   \lambda_i>0 ,$ such that $x \in S$ and $diam(S) <\e.$ One can analogously define $w^*$-SCS point in $X^*$.
\el
\edfn

We recall the following definitions from \cite{BS}. The notions of dentability, huskability and small combination of slices in closed bounded convex sets were studied in detail in \cite{GGMS} and \cite{GMS}.    

\begin{definition} \label{def bdp} A Banach space $X$ has 
	
	\begin{enumerate}
		\item { \it Ball Dentable Property} ($BDP$) if the closed unit ball has slices of arbitrarily small diameter. 
		\item {\it  Ball Huskable Property}($BHP$) if the unit ball has a relatively weakly open subset of arbitrarily small diameter.
		\item {\it Ball Small Combination of Slice Property} ($BSCSP$) if the unit ball has a convex combination of slices of arbitrarily small diameter.
	\end{enumerate} 
	
\end{definition}
\begin{remark}
Analogously, we can define $w^{*}$-$BDP$, $w^{*}$-$BHP$ and $w^*$-$BSCSP$ in a dual space by taking $w^*$slices, $w^*$-open sets and $w^*$- small combination of slices respectively. 
\end{remark}
 %Observe that for a Banach space , $BDP$ always implies $BHP$, in fact, any slice of the unit ball is relatively weakly open. 
%Again, by Bourgain's Lemma \cite{GGMS}, every non-empty relatively weakly open subset of $B_X$ contains a finite convex combination of slices, so $BHP$ implies $BSCSP$. %Similar observations are true for $w^*$-versions. Since every $w^*$-slice ($w^*$-open set) of $B_{X^*}$ is also a slice (weakly open set) of $B_{X^*}$, so we have the following diagram :
It is easy to observe the following :
$$ BDP \Longrightarrow \quad BHP \Longrightarrow \quad  BSCSP$$ $\quad \quad \quad \quad \quad \quad \quad \quad\quad\quad\quad\quad\quad\quad\quad\quad \Big \Uparrow \quad \quad\quad\quad\quad \Big \Uparrow \quad \quad\quad\quad\quad \Big \Uparrow$  $$ w^*BDP \Longrightarrow  w^*BHP \Longrightarrow  w^*BSCSP$$
%$$  diagram  -\quad 1$$
In general, none of the reverse implications hold good. For details, see \cite{BS}.
We also have  the following result from \cite{BS} which we will need in our subsequent discussion. 

\begin{proposition} \label{A1}
	\cite{BS} A Banach space $X$ has $BDP$ (resp. $BHP$ , $BSCSP$) if and only if $X^{**}$ has $w^*$-$BDP$ (resp. $w^*$-$BHP$, $w^*$-$BSCSP$). 
\end{proposition}

In this work, we study these properties in ideals of Banach spaces. The notion of ideals in Banach spaces,
 introduced by Godefroy, Kalton and Saphar in \cite {GKS}, continues to be a very important concept in geometry of Banach space. We include here some of the references relevant to our work, namely, \cite{HWW}, \cite{W}, \cite{R1}, \cite{KKO},\cite{HL} and \cite{ALN}.

 We recall the following definitions and results which we need in our discussion:

 \begin{definition}Let $Y\subseteq X$ be a subspace of $X.$ The annihilator of $Y$ in the dual space  $X^*$ is the subspace of $X^*$ defined by 
 $Y^\perp = \Big \{ x^*\in X^* : x^*(y) = 0\quad \forall y\in Y \Big \}$
 \end{definition}
\begin{definition}\cite{GKS}
Let $Y$ be a closed subspace of $X.$ Then  $Y$ is said to be an  ideal in $X$ if $Y^{\perp}$ is the kernel of a norm one projection in $X^*.$
\end{definition}
 \begin{remark}
 	\begin{enumerate}
 		\item Suppose $P: X^\ast \rightarrow X^\ast$ is projection 
 		with $\|P\| = 1$, $ker(P) = Y^\bot.$ For $x^* \in X^*$ since $P(x^*)-x^*=0$ on $Y$ , as $\|P\|=1$ , we see that $P(x^*)$ is a norm preserving extension of $x^*|_Y.$
 
 \item If we look at $X$ as an embedding in $X^{**}$ via the map $J :X \rightarrow  X^{**}$ defined by
 $J(x)(x^{*}) = x^{*}(x),$ 
 the natural projection on $X^{***}$ by $\Gamma \rightarrow \Gamma|_X$ is a norm one linear projection with kernel $X^\perp.$ Thus $X$ is isometric to an ideal in $X^{**}$ via this canonical embedding $J.$ 
\end{enumerate}
 \end{remark}
  
\begin{definition}
	Let $X$ be a Banach space. A linear projection $P$ on $X$ is called 
	\begin{enumerate}
		\item an $L$-projection if $\Vert x \Vert = \Vert Px \Vert + \Vert x-Px \Vert$ for all $x \in X.$
		\item an $M$-projection if $\Vert x \Vert = Max \{\Vert Px \Vert, \Vert x-Px \Vert \}$ for all $x \in X.$
	\end{enumerate}
	
\end{definition}

\begin{definition}
	Let , $X$ be a Banach space . A closed subspace $Y \subset X$ is called an $M$-ideal if there exists an $L$-projection $P$ : $X^* \rightarrow X^*$ with Ker $P$ = $Y^\perp$  
\end{definition}

We recall from Chapter I of \cite{HWW} that when $Y \subset X$ is an
$M$-ideal, elements of $Y^\ast$ have unique norm-preserving
extension to $X^\ast$ and one has the identification, $X^\ast =
Y^\ast \oplus_1 Y^\perp$.
Several examples from among
function spaces and spaces of operators that satisfy these geometric
properties can be found in the monograph \cite{HWW}. See also \cite{KKO}.

\begin{definition} An ideal $Y \subseteq X $ is said to
	be a strict ideal if for a projection $P: X^\ast \rightarrow X^\ast$
	with $\|P\| = 1$, $ker(P) = Y^\bot$, one also has  $B_{P(X^\ast)}$ is $w^*$-dense in $B_{X^*}$ or
	in other words $B_{P(X^*)}$ is a norming set for $X.$ 
\end{definition}
In the case of an
ideal also one has that $Y^\ast$ embeds (though there may not be
uniqueness of norm-preserving extensions) as $P(X^\ast)$. Thus we
continue to write $X^* = Y^* \oplus Y^\perp.$ In what follows we use
a result from \cite{R1}, that identifies strict ideals as those for
which $Y \subset X \subset Y^{\ast\ast}$ under the canonical
embedding of $Y$ in $Y^{\ast\ast}.$
A prime example of a strict ideal is a Banach space $X$ under its canonical embedding in $X^{\ast\ast}.$

\begin{definition}
An operator $f : Y^* \rightarrow X^*$ is called a  Hahn Banach extension operator  if it satisfies following two conditions : 
\begin{enumerate}
\item $(fy^*)|_Y = y^*$ $\forall y^* \in Y^*$
\item $\Vert fy^* \Vert = \Vert y^* \Vert $ $\forall y^* \in Y^*$
\end{enumerate}
\end{definition}
%\begin{proposition}\label{ideal}
%\cite{OP} A subspace $Y$ of $X$ is an ideal if and only if there exists a Hahn Banach extension operator $f : Y^* \rightarrow  X^*$ such that for every $\varepsilon > 40$ and every finite dimensional subspace $E\subset X$ , finite dimensional subspace $F\subset Y^*$ , there exists $T: E \rightarrow Y$ such that followings hold :
%\begin{enumerate}
%\item $Te = e $  $\forall e \in E\bigcap Y$
%\item $ \Vert Te \Vert \leqslant (1+\varepsilon) \Vert e \Vert$  $\forall e\in E$
%\item $fy^*(e) = y^*(Te)$  $ \forall e\in E$ $\forall y^* \in F$  
%\end{enumerate}  
%\end{proposition}
\begin{definition}
A subspace $Y$ of $X$ is called  an almost isometric ideal ($ai$-ideal) in $X$ if for every $\varepsilon >0$ and every finite dimensional subspace $E\subset X$ there exists $T: E\rightarrow Y$ such that followings hold :
\begin{enumerate}
\item $Te = e $  $\forall e \in E\bigcap Y$
\item $\frac{1}{1+\varepsilon} \Vert e \Vert \leqslant \Vert Te \Vert \leqslant (1+\varepsilon) \Vert e \Vert$  $\forall e\in E$
\end{enumerate}
\end{definition}
Almost Isometric ideals were first introduced in \cite{ALN}. It is easy to see that the notion of  $ai$-ideal is a direct generalisation of Principle of local reflexivity. For more details, see \cite{JL}, \cite{OP}.

\begin{proposition} \label{ai}
\cite{ALN}  A subspace $Y$ of $X$ is $ai$-ideal if and only if there exists a Hahn Banach extension operator $f : Y^* \rightarrow  X^*$ such that for every $\varepsilon > 0$ and every finite dimensional subspace $E\subset X,$ finite dimensional subspace $F\subset Y^*$ , there exists $T: E \rightarrow Y$ such that followings hold :
\begin{enumerate}
\item $Te = e $  $\forall e \in E\bigcap Y$
\item $\frac{1}{1+\varepsilon} \Vert e \Vert \leqslant \Vert Te \Vert \leqslant (1+\varepsilon) \Vert e \Vert$  $\forall e\in E$
\item $fy^*(e) = y^*(Te)$  $ \forall e\in E$ $\forall y^* \in F$  
\end{enumerate}  
\end{proposition}

We show  that if a Banach space $X$ has $BHP,$ then any $M$-ideal of $X$ also has $BHP.$ We further show that if $Y$ is an $M$-ideal of $X,$ then $X^*$ has $w^{*}$-$BHP$ whenever $Y^*$ has $w^{*}$-$BHP.$
 We use these results to prove that for a compact Hausdorff space $K,$ which has an isolated point, $X$ has $BHP$ whenever $C(K,X)$ has $BHP$ and $C(K,X)^*$ has $w^{*}$-$BHP$ whenever $X^*$ has $w^{*}$-$BHP.$
We also prove that, if $Y$ is a strict ideal of $X,$ then $X^*$ has $w^{*}$-$BHP$ whenever $Y^*$ has $w^{*}$-$BHP.$
Lastly, we show that if $Y$ is an almost isometric ideal of $X$ then $Y$ has $BHP$ implies $X$ also has $BHP.$
We obtain similar results for $BDP$ and $BSCSP.$  
The spaces that we will be considering have been well studied in literature. A large class of function spaces like the Bloch spaces, Lorentz and Orlicz spaces, spaces of vector valued functions and spaces of compact operators are examples of the spaces we will be considering, for details, see \cite{HWW}. In some of our proofs, we use similar techniques as in \cite{ALN}and \cite{HL}.

\section{SMALL DIAMETER PROPERTIES IN IDEALS OF BANACH SPACES. }
We need the following result for our discussion

The following Proposition is an easy consequence of Proposition 2.3 \cite{W}. 

\begin{proposition} \label{lem Mideal}
	 Let $Y$ be an $M$-ideal in $X$. Then, for all $x\in B_X$ and for all $\delta >0,$ there exists a net $(y_\alpha)_\alpha$ in $Y$ such that $(y_\alpha)$ converges to $x$ in $\sigma (X,Y^*)$ topology. Moreover, for $y_0{^*}\in B_{Y^*}$ there exists $z\in B_Y$ such that 
	 \begin{enumerate}
	 \item $\Vert y+x-z\Vert < (1+\delta)$ \quad $\forall y\in B_Y$ and 
	 \item $\vert z(y_0{^*})-x(y_0{^*})\vert <\delta$
	\end{enumerate}
\end{proposition}

\begin{proposition}\label{prop 1}
	Let $Y\subset X $ be an $M$-ideal. Then  $X$ has  $BSCSP$  implies $Y$ has $BSCSP$  .
\end{proposition}

\begin{proof}
	%We will prove this result by contradiction.
	Suppose $Y$ does not have $BSCSP.$  Then there exists $\varepsilon >0,$ such that every convex combination of slices of $B_Y$ has diameter greater than $ \varepsilon.$ 
	Since $X$ has $BSCSP,$  there exists a convex combination of slices $\sum_{i=1}^{n} \lambda_i S(B_X,x_i^* , \alpha_i)$ of $B_X$ with diameter less than $ \frac{\varepsilon}{1+\varepsilon}.$
	Choose $\alpha$ such that 
	%$0<\alpha <\varepsilon$ and 
	$0<\alpha<\min\{\varepsilon,\frac{\alpha_1}{3},\frac{\alpha_2}{3},\ldots,\frac{\alpha_n}{3}\}.$
	Since $Y$ is an $M$-ideal, there exists an $L$-projection $P:X^*\rightarrow X^*$ with $ker(P)=Y^\perp.$ 
	If  $Px_i^* = 0$ we choose  $\hat{y}_i^*\in S_{Y^*}$ arbitrarily and fix it. %and therefore 
	%Choose arbitrary $y_i^*\in S_{Y^*}$ for all i with $Px_i^* = 0.$\\   
	%Therefore, 
	%for each i ($1\leqslant i \leqslant n$) 
	Now define 
	% $$y_i^* = \frac{Px_i^*}{\Vert Px_i^* \Vert }\quad and \quad \beta_i = \frac{ \alpha - \alpha \Vert Px_i^* \Vert + \alpha^2}{\Vert Px_i^* \Vert }$$
	
	$\vspace{.5 cm}$
	
	$y_i^*=    \left\{ \begin{array}{rcl}
		\frac{Px_i^*}{\Vert Px_i^* \Vert } &\mbox{for}
		& Px_i^* \neq 0 \\
		\hat{y}_i^* &\mbox{for}
		& Px_i^* = 0
	\end{array} \right.$
	and 
	$\beta_i=     \left\{ \begin{array}{rcl}
		\frac{ \alpha - \alpha \Vert Px_i^* \Vert + \alpha^2}{\Vert Px_i^* \Vert }  &\mbox{for}
		& Px_i^* \neq 0 \\
		\alpha &\mbox{for}
		& Px_i^* = 0
	\end{array} \right.$ 
	
	$\vspace{.5 cm}$
	
	Therefore , $\sum_{i=1}^{n} \lambda_i S(B_Y,y_i^* , \beta_i)$ is a convex combination of slices of $B_Y,$ by assumption it has diameter greater than $\varepsilon.$  Hence, there exist $y_1^1,y_2^1,\ldots,y_n^1\in B_Y$ \ and  \ $y_1^2,y_2^2,\ldots,y_n^2 \in B_Y$ such that 
	$$\Vert \sum_{i=1}^{n} \lambda_i (y_i^1 - y_i^2) \Vert > \varepsilon$$
	$$ and\quad  y_i^* (y_i^k) > 1 - \beta_i \quad \forall k=1,2 \quad \forall i= 1,2,...,n$$
	
	Now for $Px_i^* \neq 0$ 
	$$ \frac{Px_i^*}{\Vert Px_i^* \Vert } (y_i^k) > 1 - \beta_i \quad \forall k=1,2  $$ 
	$$\Rightarrow Px_i^* (y_i^k) > (\Vert Px_i^* \Vert - \alpha)(1+\alpha)\quad \forall k=1,2 $$
	Also there exists $x_1,x_2,\ldots,x_n \in B_X$ such that
	$$(x_i^* - Px_i^*)(x_i) > (\Vert x_i^* - Px_i^* \Vert - \alpha )(1+\alpha)\quad \forall i= 1,2,...,n$$
	Since $Y$ is an $M$-ideal in $X,$ by Proposition  $\ref{lem Mideal}$ for every $i=1,2,...,n$, there exist $z_i \in B_Y$ such that $$\Vert y_i^k + x_i - z_i \Vert < 1+\alpha \quad \forall k=1,2 $$ $$ and \quad  \vert Px_i^* (x_i - z_i)\vert < \alpha $$
	Put $$x_i^k = \frac{y_i^k + x_i - z_i}{1+\alpha} \quad \forall  k=1,2 \quad \forall  i= 1,2,...,n$$
	Thus $x_i^k \in S(B_X , x_i^*, \alpha_i)$ $\forall k=1,2$ . \\
	Indeed , $$x_i^* (x_i^k) = \frac{x_i^*(y_i^k + x_i - z_i)}{1+\alpha} \hspace{5 cm}$$  
	$$= \frac{Px_i^*(y_i^k) + (x_i^* - Px_i^*)(x_i) + Px_i^* (x_i - z_i)}{1+\alpha}$$ 
	%   $$> \Vert Px_i^*\Vert - \alpha + \Vert x_i^* - Px_i^* \Vert - \alpha - \alpha \hspace{1.3 cm}$$  
	%    $$= \Vert x_i^* \Vert - 3\alpha \hspace{4.8 cm}$$  
	% $$= 1- 3\alpha \hspace{5.2 cm}$$ 
	$$\hspace{1.5 cm}>    \left\{ \begin{array}{rcl}
		\Vert Px_i^*\Vert - \alpha + \Vert x_i^* - Px_i^* \Vert - \alpha - \alpha 
		&\mbox{for} & Px_i^* \neq 0 \\
		\|x_i^*\| -\alpha 
		&\mbox{for} & Px_i^* = 0
	\end{array} \right.$$
	$$\geqslant    \left\{ \begin{array}{rcl}
		\Vert x_i^* \Vert - 3\alpha
		&\mbox{for} & Px_i^* \neq 0 \\
		\|x_i^*\| -\alpha 
		&\mbox{for} & Px_i^* = 0
	\end{array} \right. \hspace{1.3 cm}$$
	$$=    \left\{ \begin{array}{rcl}
		1 - 3\alpha
		&\mbox{for} & Px_i^* \neq 0 \\
		1 -\alpha 
		&\mbox{for} & Px_i^* = 0
	\end{array} \right.$$ 
	$$> 1-\alpha_i \hspace{3 cm}$$ 
	Now , $\Vert \sum_{i=1}^{n} \lambda_i  (x_i^1 - x_i^2) \Vert  =  \frac{\sum_{i=1}^{n} \lambda_i  (y_i^1 - y_i^2)}{1+\alpha}   >\frac{\varepsilon}{1+\alpha} >\frac{\varepsilon}{1+\varepsilon} $ \ ,
	\  a \  contradiction.\\
	Thus $Y$ has $BSCSP$ .
\end{proof}
However for $BDP,$ we have the following:
 %We cannot take the assumption as in Proposition $\ref{prop 1}$ for the case $BDP$ unless $Y$ is an $M$-summand.
\begin{proposition} \label{ne}
	If $X$ has a proper $M$ ideal then $X$ fails to have $BDP.$ 
\end{proposition}
\begin{proof}
	Suppose $Y$ is a proper $M$-ideal of $X.$ Thus $X^*= Y^*\oplus_1 Y^{\perp}.$ If possible let, $X$ have $BDP.$ By Proposition $\ref{A1}$, $X^{**}$ has $w^*BDP.$ It now follows from Proposition 2.17, \cite{BS}, $Y^{**}$ has $w^*BDP.$  Applying Proposition $\ref{A1}$ again we can conclude that $Y$ has $BDP,$ a contradiction, since a proper $M$-ideal fails to have $BDP$ (see \cite{HWW}, Chapter II, Theorem 4.4).
	% (by Proposition $\ref{A1}$)
	% which implies $Y^{**}$ has $w^*BDP$ (by Proposition $\ref{1 sum w star case}$). As a result $Y$ has $BDP,$ a contradiction due to the Proposition $\ref{HWW}.$
\end{proof}
In view of Proposition $\ref{ne}$ and Proposition $\ref{prop 1},$ we have the following remark 
%by considering n=1 in  Proposition $\ref{prop 1}.$
\begin{remark} \label{M ideal BDP}
	Putting $n=1$ in Proposition $\ref{prop 1},$ it follows that if $X$ is a Banach space with $BDP$ and $Y$ is an $M$-summand of $X,$ then $Y$ has  $BDP.$  
	This also follows from Proposition 2.7, \cite{BS}.
	
	%Let $X$ be a Banach space with $BDP$ and $Y$ an $M$-summand of $X.$ Then by considering n=1,  we can conclude from Proposition $\ref{prop 1}$ that  $Y$ also has  $BDP.$ 
\end{remark}

\begin{proposition}\label{bhp mideal}
	Let $X$ be a Banach space and let $Y\subset X $ be an $M$-ideal , then  $X$ has  $BHP$  implies $Y$ has  $BHP.$
\end{proposition}

\begin{proof}
	 Suppose $Y$ fails  $BHP.$  Then there exists $\varepsilon >0$ such that every nonempty relatively weakly open subset of $B_Y$ has diameter greater than $ \varepsilon.$
	Since $X$ has $BHP,$ there exists a basic relatively weakly open subset $U=\{x\in B_X  : \vert x_i^* (x-x_0)\vert <\gamma , i= 1,2,...,n\} \quad (x_i^* \in S_{X^*} \forall i \quad and \quad x_0 \in B_X)$ with diameter less than  $\frac{\varepsilon}{1+\varepsilon}.$
	Choose $\beta >0 $ such that $\beta  < min \{ \frac{\gamma}{4}, \varepsilon\}.$ 
	Since $Y$ is an $M$-ideal of $X,$ there exists a $L$-projection $P:X^*\rightarrow X^*$ with $ker(P)=Y^\perp.$
	Also, by Proposition  $\ref{ Mideal}$ there exists $y_0 \in B_Y$ such that $\Vert y+x_0-y_0\Vert < 1+\beta \quad \forall y\in B_Y.$
	Consider, $V=\{y\in B_Y  : \vert Px_i^* (y-y_0)\vert <\beta , i= 1,2,...,n\}.$ 
	Then $V$  is a nonempty relatively weakly open subset of $B_Y,$ so $dia(V)>\varepsilon.$ Hence there exists $y,\tilde{y} \in V$ such that $\Vert y-\tilde{y} \Vert > \varepsilon.$
	Now  $\Vert y+x_0-y_0\Vert < 1+\beta $ and $\Vert \tilde{y} +x_0-y_0\Vert < 1+\beta. $
	Put $$ x= \frac{y+x_0-y_0}{1+\beta} \quad and \quad  \tilde{x} = \frac{\tilde{y} +x_0-y_0}{1+\beta}.$$
	So, for $i = 1,2,..,n,$
	$$\vert x_i^*(x-x_0)\vert \quad = \quad \frac{\vert x_i^*(y-\beta x_0 - y_0)\vert}{1+\beta} \quad \leqslant \frac{\vert Px_i^*(y - y_0)\vert + \beta \vert x_i^* (x_0) \vert}{1+\beta}$$  
	$$\quad\quad\quad\quad\quad\quad\quad\quad\quad\quad
	\quad\quad\quad\quad\quad < \frac{2\beta}{1+\beta} <2\beta <\gamma$$
	Thus , $x\in U$ and similarly $\tilde{x} \in U$\\
	Now, $\Vert x-\tilde{x}\Vert = \frac{\Vert y-\tilde{y}\Vert}{1+\beta} > \frac{\varepsilon}{1+\beta}>\frac{\varepsilon}{1+\varepsilon},$
	 a contradiction.
	Thus $Y$ has $BHP.$ 
\end{proof}

\begin{proposition}\label{mideal}
	If $Y\subset X$ is an $M$-ideal in $X$ , then $Y^*$ has $w^*BHP$ implies $X^*$ has $w^*BHP.$
\end{proposition}
\begin{proof}
	Let $\varepsilon>0.$ Since $Y^*$ has $w^*BHP$, $B_{Y^*}$ has $w^*$ open subset of diameter less than $\varepsilon.$ Choose $m_0^*$ in $S_{Y^*}\bigcap V.$
	Then there exists $V_0 = \{ y^* \in B_{Y^*} : \vert y^*(y_i) - y_0^*(y_i)\vert <\alpha \forall i=1,2,...n\}\subset V$ for some $n\in \mathbb{N}$ and $y_1,y_2,\ldots, y_n\in B_Y.$
	%Now we construct $w^*$ open set $U_0$ of $B_{M^*}$ such that $\Vert m^* \Vert > 1-\varepsilon $ $\forall m^*\in U_0.$
	Since $Y$ is an $M$-ideal, $X^*= Y^* \oplus_1 Y^\perp.$
	For $y_0^*\in S_{Y^*},$ there exists  $y_0\in B_Y$ such that $\vert y_0^*(y_0)\vert > 1-\varepsilon.$
	Choose $\gamma>0$ such that $\vert y_0^*(y_0)\vert > 1-\varepsilon+\gamma.$
	Let $U_0=\{y^* \in B_{Y^*} : \vert y_0^*(y_0) - y^*(y_0)\vert <\gamma \}.$
	Then, for $y^* \in U_0, \vert y^*(y_0)\vert > \vert y_0^*(y_0)\vert -\gamma > (1-\varepsilon + \gamma)-\gamma=1-\varepsilon.$
	Choose $0<\delta<\min\{\alpha , \gamma\}.$
	Let, $$W=\{x^*\in B_{X^*} : \vert x^*(y_i) - y_0^*(y_i)\vert<\delta , i=0,1,2,\ldots,n\}$$
	Clearly , $W$ is a relatively $w^*$ open subset of $B_{X^*}.$
	Then , $W\subset V_0 + \varepsilon B_{Y^\perp}$
	Indeed , let $x^*\in W.$ Then there exists $y^*\in Y^*$ and $y^\perp \in Y^\perp$ such that $x^*=y^*+y^\perp.$\\
	Then $$\vert x^*(y_i)-y_0^*(y_i)\vert <\delta \quad \forall i=0,1,2,\ldots, n. $$
	$$\Rightarrow \vert (y^*+y^\perp)(y_i)-y_0^*(y_i)\vert <\delta \quad \forall i=0,1,2,\ldots,n. $$
	$$\Rightarrow \vert y^*(y_i)-y_0^*(y_i)\vert <\delta \hspace{1.8 cm} \forall i=0,1,2,\ldots,n. $$
	Hence, $\vert y^*(y_i)-y_0^*(y_i)\vert <\delta<\alpha \quad \forall i=0,1,2,\ldots,n $ \ and \ $\vert y^*(y_0)-y_0^*(y_0)\vert <\delta<\gamma$\\
	So , $y^*\in V_0$ and $y^*\in U_0$ , which implies $y^*\in V_0$ and $\Vert y^* \Vert > 1-\varepsilon .$\\
	Since, $\Vert x^* \Vert = \Vert y^* \Vert + \Vert y^\perp \Vert, $ it follows that $\Vert y^\perp \Vert < \varepsilon$ \quad i.e.\quad $y^\perp \in \varepsilon B_{Y^\perp}.$
	Thus $x^*=y^*+y^\perp \in V_0 + \varepsilon B_{Y^\perp}.$
	Hence , $$diam(W)\leqslant diam(V_0) + diam ( \varepsilon B_{Y^\perp})\hspace{1.5 cm}$$
	$$\leqslant diam(V) + diam ( \varepsilon B_{Y^\perp})$$
	$$<\varepsilon + \varepsilon\hspace{3 cm}$$
	$$<2\varepsilon\hspace{3.2 cm}$$
	
\end{proof}

	\begin{remark}
		Similarly, for an  $M$- ideal $Y$ in   $X$,   $Y^*$ has $w^*BSCSP$ ( $w^*BDP$ ) implies $X^*$ has $w^*BSCSP$ ( $w^*BDP$), see \cite{BR},\cite{B2}.
	\end{remark}

We recall that for a compact Hausdorff space
$K$, $C(K,X)$ denotes the space of continuous $X$-valued functions
on $K$, equipped with the supremum norm. We recall from \cite{L} that dispersed compact Hausdorff spaces have isolated points.

\begin{proposition}\label{c(k)}
	Let $K$ be a compact, Hausdorff space with an isolated point. Then
	\blr
		\item If  $C(K,X)$  has $BSCSP$(respectively $BHP$, $BDP$)then $X$ has $BSCSP$ (respectively $BHP$, $BDP$).
		\item If $X^*$ has $w^*$-$BHP$ , then  $C(K,X)^*$  has $w^*$-$BHP$.
		\el
\end{proposition}

\begin{proof}
	Let $k_0\in K$ be a isolated point. Then the map $P$ such that  $F \longmapsto \chi_{k_0} F$ is an $M$-projection in $C(K,X)$ whose range is isometric to $X.$  
	Thus $X$ is isometric to an $M$-summand and hence an $M$-ideal in $C(K,X).$
	(i) follows from Proposition 2.15 \cite{BS}, Remark \ref{mideal bdp}  and Proposition~ \ref{bhp mideal}.
	(ii) follows from Proposition \ref{mideal}.
	\end{proof}

\begin{remark}
	Similar results as in (ii), also hold good for $w^*$-$BDP$ and $w^*$-$BSCSP$. See \cite {B2} and \cite {BR}.
\end{remark}

\begin{proposition}
	Let $Y$ be a strict ideal of  $X.$ If $Y^*$ has $w^*BHP,$ then $X^*$ has $w^*BHP$.
\end{proposition}

\begin{proof}
	Let $\varepsilon > 0.$ Since $Y^*$ has $w^*BHP$ then $B_{Y^*}$ has a basic relatively $w^*$ open subset
	$V=\{y^* \in B_{Y^*} : \vert y^*(y_i) - y_0^*(y_i)\vert < \alpha , i= 1,2,...,n\}$ where $y_1,y_2,...y_n \in B_Y$
	of diameter less than $\varepsilon.$
	Consider, relatively $w^*$ open subset of $B_{X^*}$ as, 
	$$W= \{x^* \in B_{X^*} : \vert x^*(y_i) - y_0^*(y_i)\vert < \alpha , i= 1,2,...,n\}$$
	Since $Y$ is a strict ideal in $X,$ we have $B_{X^*}= \overline{B_{Y^*}}^{w^*}$ , hence we have $W \subset \overline{V}^{w^*}$ . Indeed , let $x_0^* \in W$. Then there exists net $(x_\lambda^*)$ in $B_{Y^*}$ that converges to $x_0^*.$  \\
	Now, $$\displaystyle{\lim_{\lambda}} \vert x_\lambda^*(y_i) - y_0^*(y_i)\vert = \vert x_0^*(y_i) - y_0^*(y_i)\vert \quad  \forall i = 1 , 2 ,..., n $$ 
	Since $x_0^* \in W$ so we get a $\lambda_0$ such that $\vert x_\lambda^*(y_i) - y_0^*(y_i)\vert <\alpha$  $\forall \lambda \geqslant \lambda_0$ $\forall i=1,2,..,n.$
	Hence $x_\lambda^* \in V$   $\forall \lambda \geqslant \lambda_0$ and therefore, $x_0^* \in \overline{V}^{w^*}.$
	Thus, $dia(W) \leqslant dia (\overline{V}^{w^*}) = dia (V) < \varepsilon.$
\end{proof}
\begin{remark}
	Similar results hold good for $w^*BSCSP$ and  $w^*BDP.$ For details, see  \cite{BR},  \cite{B2}.
	\end{remark}

However, the converse of these results are not true. We give an example. 
\begin{example}\label{example}
	Let $X=C[0,1].$ Since  $X$ is always a  strict ideal in $X^{**}, C[0,1]$ is also so in its double dual. It is well known that $X^*=L_1[0,1] \oplus _1 Z$, for some subspace $Z$ of $X^*$ with $RNP$ and hence $Z$ has $BDP.$ Since  one component of $X^*$ i.e., $Z$ has $BDP,$ it follows from Proposition 2.4 \cite{BS}, that  $X^*$ has $BDP$ and hence it has $BHP$ and $BSCSP.$  
	We also know from \cite{BS} that a Banach space $X$ has $BDP$ (resp. $BHP$ , $BSCSP$) if and only if $X^{**}$ has $w^*$-$BDP$ (resp. $w^*$-$BHP$, $w^*$-$BSCSP$). 
	Hence  $X^{***}$ has  $w^*$-$BDP$(resp. $w^*$-$BHP$, $w^*$-$BSCSP$). 
	It is also known that every convex combination of $w^*$ slices of $B_{X^*}$ has diameter two ( 
	see  \cite{BGLPRZ} for details). Hence, $X ^*$ does not have $w^*$-$BSCSP$ and consequently it does not have $w^*$-$BHP$ and $w^*$-$BDP.$
	\end{example}

%\section{Almost Isometric Ideals and Small Diameter Properties}
%In this section we show that $BDP$, $BHP$ and $BSCSP$ can be lifted from an almost isometric ideal to the whole space.
\begin{proposition}\label{bscsp ai ideal}
If $Y$ is an $ai$-ideal in $X$ then $Y$ has $BSCSP$ implies $X$ has $BSCSP.$
\end{proposition}
\begin{proof}
%We use similar techniques as in \cite{ALN}.\\
Suppose $X$ does not have $BSCSP.$ Then there exists $\varepsilon > 0,$  such that every convex combination of slices of $B_X$ has diameter greater than $\epsilon.$  Since $Y$ has $BSCSP$ , for $\e> 0$ there exists a convex combination of slices, $S=\sum_{i=1}^{n} \lambda_i S_{i}(B_Y,y_i^*,\alpha_i)$ of $B_Y,$
where  $\alpha_i > 0$ $\forall i=1,2,..n$ , $1 \geq \lambda i > 0, \sum_{i=1}^{n} \lambda_i =1 ,y_i^*\in S_{Y^*}$  $\forall i=1,2,...,n$  and $diam(S)<\frac{\varepsilon}{4}.$
Since $Y$ is an $ai$-ideal in $X,$ there exists a  Hahn Banach extension operator, $f : Y^* \rightarrow  X^*$ satisfying the conditions of Proposition $\ref{ai}$.
Consider $\bar{S}=\sum_{i=1}^{n} \lambda_i S_i(B_X,fy_i^*,\alpha_i).$
Since $\bar{S}$ is a convex combination of slices of $B_X$,  $diam(\tilde{S}) >\varepsilon.$ Hence there exists $x=\sum_{i=1}^{n} \lambda_i x_i$ , $\tilde{x}=\sum_{i=1}^{n} \lambda_i \tilde{x_i} \in \bar{S}$ such that $\Vert x - \tilde{x}\Vert >\varepsilon.$
Now we can choose $0<\alpha<1$  such that $\frac{x_i}{1+\alpha} \in S_i(B_X,fy_i^*,\alpha_i)$ and $\frac{\tilde{x_i}}{1+\alpha} \in S_i(B_X,fy_i^*,\alpha_i)$ $\forall i=1,2,...,n.$
Put $x'=\sum_{i=1}^{n} \lambda_i \frac{x_i}{1+\alpha}=\frac{x}{1+\alpha}$ , $\tilde{x}'=\sum_{i=1}^{n} \lambda_i \frac{\tilde{x_i}}{1+\alpha}=\frac{\tilde{x}}{1+\alpha} \in \bar{S}.$
Also , $\Vert x' - \tilde{x}'\Vert =\frac{\Vert x - \tilde{x}\Vert}{1+\alpha}>\frac{\varepsilon}{1+\alpha}$.
Put $E= span\{x',\tilde{x}',\frac{x_1}{1+\alpha},\frac{x_2}{1+\alpha},...\frac{x_n}{1+\alpha},\frac{\tilde{x_1}}{1+\alpha},\frac{\tilde{x_2}}{1+\alpha},...\frac{\tilde{x_n}}{1+\alpha}\}$ and $F= span\{y_1^*,y_2^*,...y_n^*\}.$
For  $E,F$ and $\alpha$ there exists   $T:E\rightarrow Y$ satisfying the conditions in Proposition $\ref{ai}.$ %Now, $Tx',T\tilde{x}'\in Y.$
By  Proposition $\ref{ai}$ (ii), $\Vert Tx'\Vert\leqslant (1+\alpha)\Vert x'\Vert\leqslant 1.$ Similarly $\Vert T\tilde{x}'\Vert\leqslant 1.$\\
Again, by  Proposition $\ref{ai}$ (ii), $$\Vert Tx'-T\tilde{x}'\Vert \geqslant \frac{\Vert x'-\tilde{x}'\Vert}{1+\alpha}>\frac{1}{1+\alpha} \frac{\varepsilon}{1+\alpha}>\frac{\varepsilon}{4}$$
Also, $\forall i=1,2,...,n,$ by, Proposition $\ref{ai}$ (iii),
$$ y_i^*(T\frac{x_i}{1+\alpha})=fy_i^*(\frac{x_i}{1+\alpha})>1-\alpha_i$$
Which implies  $Tx'\in S.$ Similarly $T\tilde{x}'\in S,$ a contradiction as $diam \ S<\frac{\varepsilon}{4}$
\end{proof}
Arguing similarly as above for $n=1$, we have, 
\begin{corollary}\label{bdp ai ideal}
If $Y$ is an $ai$-ideal in $X$ then $Y$ has $BDP$ implies $X$ has $BDP.$
\end{corollary} 
\begin{proposition}\label{bhp ai ideal}
If $Y$ is an $ai$-ideal in $X$ then $Y$ has $BHP$ implies $X$ has $BHP.$
\end{proposition}
\begin{proof}
%We use similar techniques as in \cite{ALN}.\\
Suppose  $X$ does not have  $BHP.$ Then there exists $\varepsilon > 0$ such that every relatively weakly open subset of $B_X$ has diameter greater than $\varepsilon.$ Since $Y$ has $BHP$, for  $\varepsilon > 0$,  $B_Y$ has a relatively weakly open subset $U=\{y\in B_Y : \vert y_i^*(y-y_0)\vert <\delta , i=1,2,...,n\}$
where $y_0\in B_Y$, $\delta> 0$, $y_i^*\in Y^*$  $\forall i=1,2,...,n$  and $diam \ (U)<\frac{\varepsilon}{4}.$  Since $Y$ is an $ai$-ideal in $X,$ there exists a  Hahn Banach extension operator $f : Y^* \rightarrow  X^*$ as in Proposition $\ref{ai}$.
Consider $V=\{x\in B_X : \vert fy_i^*(x-y_0)\vert <\delta , i=1,2,..,n\}.$
$V$ is a relatively weakly open subset of $B_X,$ hence $diam \ (V) >\varepsilon.$ So, there exists $x_1,x_2\in V$ such that $\Vert x_1 - x_2\Vert >\varepsilon.$ 
%Here we use similar technique as in \cite{ALN}.
 Choose  $0<\alpha<1$  such that 
$\tilde{x_1} = \frac{x_1}{1+\alpha} \in V$ and $\tilde{x_2} = \frac{x_2}{1+\alpha} \in V.$  Also , $\Vert \tilde{x_1} - \tilde{x_2}\Vert =\frac{\Vert x_1 - x_2\Vert}{1+\alpha}>\frac{\varepsilon}{1+\alpha}.$  Put $E= span\{y_0,\tilde{x_1},\tilde{x_2}\}$ and $F= span\{y_1^*,y_2^*,...y_n^*\}.$ 
For $E,F$ and $\alpha$ there exists   $T:E\rightarrow Y$ as in Proposition $\ref{ai}.$ %and so $T\tilde{x_1},T\tilde{x_2}\in Y.$
By  Proposition $\ref{ai}$ (ii) , $\Vert T\tilde{x_1}\Vert\leqslant (1+\alpha)\Vert \tilde{x_1}\Vert\leqslant 1.$ Similarly $\Vert T\tilde{x_2}\Vert\leqslant 1.$\\
Again by (ii) Proposition $\ref{ai}$, $$\Vert T\tilde{x_1}-T\tilde{x_2}\Vert \geqslant \frac{\Vert \tilde{x_1}-\tilde{x_2}\Vert}{1+\alpha}>\frac{1}{1+\alpha} \frac{\varepsilon}{1+\alpha}>\frac{\varepsilon}{4}$$
Also $\forall i=1,2,...,n,$ by Proposition $\ref{ai}$ (i) and (iii),
$$\vert y_i^*(T\tilde{x_1} - y_0)\vert = \vert y_i^*(T\tilde{x_1} - Ty_0)\vert = \vert fy_i^*(\tilde{x_1}-y_0)\vert<\delta$$
Thus $T\tilde{x_1}\in U.$ Similarly $T\tilde{x_2}\in U,$  a contradiction since $diam \ U<\frac{\varepsilon}{4}$
\end{proof}

\begin{Acknowledgement}
The second  author's research is funded by the National Board for Higher Mathematics (NBHM), Department of Atomic Energy (DAE), Government of India, Ref No: 0203/11/2019-R$\&$D-II/9249.
\end{Acknowledgement}

\end{document}